\documentclass[a4paper,12pt,draft]{amsart}
\usepackage[utf8]{inputenc}
\usepackage[english]{babel}
\usepackage{amsmath, amssymb, latexsym, amsfonts, bbm, stackrel,tikz,hyperref}
\usepackage[matrix,arrow,curve]{xy}
\usetikzlibrary{arrows,matrix,fit,positioning}

\newtheorem{theorem}{Theorem}
\newtheorem*{thmIntr}{Theorem}
\newtheorem{lemma}{Lemma}

\newtheorem{corollary}{Corollary}
\theoremstyle{definition}
\newtheorem{definition}{Definition}
\theoremstyle{remark}
\newtheorem{remark}{Remark}
\newtheorem*{remIntr}{Remark}

\newcommand{\sC}{{\mathcal C}}

\newcommand{\sE}{{\mathcal E}}

\newcommand{\sH}{{\mathcal H}}

\newcommand{\sS}{{\mathcal S}}

\newcommand{\sX}{{\mathcal X}}
\newcommand{\sY}{{\mathcal Y}}

\newcommand{\A}{{\mathbb A}}

\newcommand{\G}{{\mathbb G}}

\renewcommand{\P}{{\mathbb P}}
\newcommand{\Q}{{\mathbb Q}}

\newcommand{\mS}{{\mathbb S}}

\newcommand{\Z}{{\mathbb Z}}

\newcommand{\Sm}{{\mathbf{Sm}}}
\newcommand{\Spc}{{\mathbf{Spc}}}
\newcommand{\ev}{\text{ev}}
\newcommand{\Hho}{{\mathbf{Ho}}}
\newcommand{\sHom}{{\mathcal{H}{om}}}
\newcommand{\Hom}{\text{Hom}}

\newcommand{\SH}{\mathcal{SH}}

\newcommand{\id}{\operatorname{id}}

\newcommand{\Ho}{{\mathcal{H}_{\bullet}}}
\newcommand{\colim}{\operatornamewithlimits{\varinjlim}}

\newcommand{\hocolim}{\operatorname{hocolim}}

\newcommand{\Spt}{{\mathbf{Spt}}}

\begin{document}

\title{Witt sheaves and the $\eta$-inverted sphere spectrum}

\author{Alexey Ananyevskiy}
\address{
Chebyshev Laboratory\\ 
St. Petersburg State University\\ 
14th Line, 29b, Saint Petersburg\\ 
199178 Russia
}
\email{alseang@gmail.com}

\author{Marc Levine}
\address{
Universit\"at Duisburg-Essen\\
Fakult\"at Mathematik\\
Thea-Leymann-Stra{\ss}e 9\\
45127 Essen\\
Germany}
\email{marc.levine@uni-due.de}
\author{Ivan Panin}
\address{St. Petersburg Branch of V. A. Steklov Mathematical Institute,
Fontanka 27, 191023 St. Petersburg, Russia}

\email{paniniv@gmail.com}

\address{Institute for Advanced Study, Einstein Drive, Princeton, NJ, 08540, USA}

\email{panin@ias.edu}



\thanks{The author Levine thanks the SFB Transregio 45 and the Humboldt Foundation (Humboldt Professorship) for support. The authors Ananyevskiy and Panin are supported by Russian Scientific Foundation, grant N 14-21-00035.}
 
\maketitle

\begin{abstract}
Ananyevsky  has recently computed the stable operations and cooperations of rational Witt theory \cite{An15}. These computations enable us to show a motivic analog of Serre's finiteness result: 
\begin{thmIntr}  Let $k$ be a field. Then $\pi^{\A^1}_ n(\mS^- _k )_*$ is torsion for $n > 0$. 
\end{thmIntr}
As an application we define a category of Witt motives and show that rationally this category is equivalent to the minus part of $\SH(k)_\Q$.
\end{abstract} 

\tableofcontents

\section*{Introduction}

J.-P. Serre \cite{Serre} obtained the finite generation of the homotopy groups of spheres, and thereby the finiteness of all higher stable homotopy groups of spheres,  by comparing $S^n$ with the corresponding Eilenberg-MacLane space $K(\Z, n)$ and studying the resulting fiberation via the machinery of spectral sequences. In the motivic setting, finite generation results are very hard to come by, but one can still raise a weaker question: which bi-graded  homotopy sheaves of the motivic sphere spectrum $\mathbb{S}_k$ over a field $k$ are torsion? Morel's construction of the  homotopy $t$-structure on $\SH(k)$ suggests that one should organize the search by setting $\pi^{\A^1}_n(\mathbb{S}_k)_*:=\oplus_{q\in\Z}\pi^{\A^1}_{n+q,q}(\mathbb{S}_k)$. Morel's $\A^1$ connectedness theorem tells us that $\pi^{\A^1}_n(\mathbb{S}_k)_*=0$ for $n<0$ and Morel has computed $\pi^{\A^1}_0(\mathbb{S}_k)_*$ as the {\em Milnor-Witt sheaves} $\underline{K}^{MW}_{\ -*}$. 

One might therefore suspect that $\pi^{\A^1}_n(\mathbb{S}_k)_*$ is torsion for $n>0$. This is however not the case. Cisinski-D\'eglise have shown that $\pi^{\A^1}_{a,b}(\mathbb{S}_k)_\Q$ contains the sheaf of rational motivic cohomology groups $\sH^{-a}(\Q(-b))$ as a direct summand, and these latter are nonzero for all $a, b$ with $0>a\ge b$, at least in case $k$ has characteristic zero. The summand corresponding to motivic cohomology has a natural description: the symmetry involution on $\P^1\wedge\P^1$ defines an involution on $\mS_k$ and after inverting 2 decomposes $\mS_k$ into plus and minus ``eigenfactors'', $\mS_k[1/2]=\mS_k^+\oplus\mS_k^-$. The plus factor  $\mS_{k \Q}^+$ represents rational motivic cohomology. One can give an alternative description of the minus part of $\mS_k$ as $\mS_k[1/2,\eta^{-1}]$, where $\eta:\mS_k\to \Sigma^{-1,-1}\mS_k$ is the stable algebraic Hopf map. 

A. Ananyevsky has recently computed the stable operations and cooperations of rational Witt theory  \cite{An15}. These computations enable us to show a motivic analog of Serre's finiteness theorem:

\begin{thmIntr} Let $k$ be a field. Then $\pi^{\A^1}_n(\mathbb{S}_k^-)_*$ is torsion for $n>0$. 
\end{thmIntr}

\begin{remIntr} The result of Cisinski-D\'eglise mentioned above together with the vanishing of motivic cohomology in negative weights and our theorem shows that $\pi^{\A^1}_n(\mathbb{S}_k)_q:=\pi^{\A^1}_{n+q,q}(\mathbb{S}_k)$ is also torsion for $n>0$ as long as $q\ge 0$. 
\end{remIntr}

Our main theorem fills in a missing piece in the recent work of Heller-Ormsby \cite{HO}. They consider a finite Galois extension of fields $L/k$ with Galois group $G$ and define a functor from the $G$-equivariant stable homotopy category to $\SH(k)$. In case $k$ is real closed and $L=k[i]$, and  assuming the motivic analog of Serre's finiteness theorem above, they show that this functor is a fully faithful embedding; without this finiteness result, they are only able to show this after completing at a prime. 

Explicit computations of $\pi_1(\mS_k)_*(k)$ for fields of cohomological dimension $\le 3$ have been made by Ormsby-{\O}stv{\ae}r in \cite{OO}; these refine our general statement in that the torsion has bounded exponent, although the groups are still not in general finite. As our main advance here is for a field for which the Witt group is not torsion, their computations do not perhaps give enough information to support any conjecture as to the finer structure of $\pi_{a,b}(\mS_k)$, for instance, if for fixed $a>b\ge0$ this sheaf has bounded exponent, or even if for a fixed field, the group  $\pi_{a,b}(\mS_k)(F)$ has bounded exponent. Any information in this direction would certainly be of interest.\footnote{There appears to be a preliminary work of R\"ondigs-Spitzweck-{\O}stv{\ae}r that computes the stable $(1+*,*)$-stem of the motivic sphere spectrum over a field of characteristic zero.}

As another application, we define in \S\ref{sec:WittMot} a category of {\em Witt motives} and show that rationally this category is equivalent to the minus part of $\SH(k)_\Q$. 

\section{Preliminaries}

\begin{definition}
The \textit{Hopf map} is the canonical morphism of the varieties
\[
H\colon \A^2-\{0\} \to \mathbb{P}^1
\]
defined via $H(x,y)=[x,y]$. Pointing $\A^2-\{0\}$ by $(1,1)$ and $\mathbb{P}^1$ by $[1:1]$ and taking the suspension spectra we obtain the corresponding morphism
\[
\Sigma^\infty_{T} H \in Hom_{\SH (k)}(\Sigma^\infty_{T} (\A^2-\{0\},(1,1)),\Sigma^\infty_{T} (\mathbb{P}^1,[1:1])).
\]
The \textit{Hopf element} $\eta=\Sigma^{-3,-2} \Sigma^\infty_{T} H\in \mathbb{S}^{-1,-1}(pt)$ is the element corresponding to $\Sigma^\infty_{T} H$ under the suspension isomorphism and canonical isomorphisms 
\[
(\mathbb{P}^1,[1:1])\cong S^{2,1}, (\A^2-\{0\},(1,1))\cong S^{3,2}
\]
given by \cite[Lemma 3.2.15, Corollary 3.2.18 Example 3.2.20]{MV99}.
\end{definition}

\begin{definition}
Put
\[
\mathbb{S}[\eta^{-1}]=\hocolim \left(\mathbb{S} \xrightarrow{ \cup \eta} \Sigma^{-1,-1}\mathbb{S} \xrightarrow{\cup \eta} \Sigma^{-2,-2}\mathbb{S} \xrightarrow{ \cup \eta} \dots \right).
\]
\end{definition}

\begin{definition}
Let $KO$ be the $T$-spectrum representing higher Grothendieck-Witt groups (hermitian $K$-theory) constructed in \cite{PW10}. Let  
\[
KW=KO\wedge \mathbb{S}[\eta^{-1}].
\]
This spectrum inherits the structure of a $(8,4)$-periodic symplectically oriented commutative ring $T$-spectrum from $KO$. The periodicity isomorphisms are given by cup product with Bott element $\beta \in KW^{-8,-4}(pt)$.
\end{definition}

It is well known that spectrum $KW$ represents derived Witt groups as defined by Balmer \cite{Bal99} (see, for example, \cite[Theorem~5]{An12}).

\begin{theorem} \label{thm:KW_represents}
For every smooth variety $X$ and $i,j\in \Z$ there exist canonical isomorphisms $KW^{i,j}(X)\xrightarrow{\simeq} W^{i-j}(X)$. 
\end{theorem}

\begin{definition}\label{def:B}
Denote by
\[
\mathcal{B}=\left( \Sigma^{8m,4m}\beta^m \right)_{m\in \Z}  \colon \bigoplus_{m\in \Z} \Sigma^{4m}_T \mathbb{S} \to KW_\Q
\]
the morphism induced by the morphisms $\Sigma^{4m}_T \beta^m\colon \Sigma^{4m}_T \mathbb{S} \to KW_\Q$.
\end{definition}

\begin{theorem}[Theorem 13, \cite{An15}] \label{thm:rational_stable_operations_KW}
The pullback homomorphism
\[
\mathcal{B}^{KW_\Q}\colon KW_\Q^{*,*}(KW_\Q) \to KW_\Q^{*,*}(\bigoplus_{m\in \Z} \Sigma^{4m}_T \mathbb{S})
\]
is an isomorphism. In other words, the homomorphism
\[
Ev\colon KW_\Q^{*,*}(KW_\Q) \to (\prod_{m\in \Z} KW_\Q^{*,*}(pt))_h
\]
given by 
\[
Ev(\phi)=\left(\dots,\beta^2\phi(\beta^{-2}),\beta\phi(\beta^{-1}),\phi(1),\beta^{-1}\phi(\beta^{1}),\beta^{-2}\phi(\beta^{2}),\dots \right)
\]
is an isomorphism. Here the subscript $h$ denotes the subring generated by the homogeneous elements.
\end{theorem}

\begin{theorem}[Theorem 14, \cite{An15}] \label{thm:cooperations}
The pushforward homomorphism (in homology)
\[
\mathcal{B}_{KW_\Q}\colon (KW_\Q)_{*,*}(\bigoplus_{m\in \Z} \Sigma^{4m}_T \mathbb{S})\to (KW_\Q)_{*,*}(KW_\Q)
\]
is an isomorphism.
\end{theorem}

\section{Rational decomposition of $KW$}

\begin{definition}
For $m\in\Z$, let $\rho^{st}_m\in KW_\Q^{0,0}(KW_\Q)$ be the operation satisfying
\[
\rho^{st}_{m}(\beta^{n})=
\left[
\begin{array}{ll}
\beta^{n}, & n=m,\\
0, & n\neq m.
\end{array}
\right.
\]
These operations are idempotent and Theorem~\ref{thm:rational_stable_operations_KW} yields a decomposition
\[
KW_\Q=\bigoplus_{m\in \Z} KW_\Q^{(m)} 
\]
with $KW_\Q^{(m)}$ being the kernel for  $1-\rho^{st}_{m}$. 
\end{definition}

\begin{lemma} \label{lm:iso_eigenspaces}
$KW^{(m)}_\Q\cong \Sigma^{4m}_T KW^{(0)}_\Q$.
\end{lemma}
\begin{proof}
This follows from the fact that periodicity isomorphism 
\[
KW_\Q \xrightarrow{\cup \beta} \Sigma^{-4}_T KW_\Q
\]
identifies $KW_\Q^{(m)}$ with   $\Sigma^{-4}_T KW_\Q^{(m+1)}$.
\end{proof}

\begin{definition}
Let $A$ be a $T$-spectrum. The \textit{sheaf of stable $\A^1$-homotopy groups of $A$ in bi-degree $(i,j)$} is the Nisnevich sheaf $\pi^{\A^1}_{i,j}(A)$ associated to the presheaf
\[
U\mapsto A^{-i,-j}(U).
\]
We let $\pi^{\A^1}_{*,*}(A)$ denote the bi-graded sheaf $\bigoplus\limits_{i,j\in \Z} \pi^{\A^1}_{i,j}A$. Similarly, set $\pi_n^{\A^1}(A)_q:=\pi^{\A^1}_{n+q,q}(A)$ and $\pi_n^{\A^1}(A)_*:=\oplus_{q\in\Z}\pi^{\A^1}_{n+q,q}(A)$.
\end{definition}

\begin{remark}
Theorem~\ref{thm:KW_represents} yields 
\[
\pi^{\A^1}_{*,*} KW=\underline{W}[\eta,\eta^{-1},\beta,\beta^{-1}]
\]
with $\underline{W}$ being the Nisnevich sheaf associated to the presheaf of Witt groups.
\end{remark}

\begin{lemma}\label{lm:unit_iso}
The unit morphism $u_{KW}\colon \mathbb{S}\to KW$ induces an isomorphism
\[
(u_{KW})_*\colon \bigoplus_{n\in \Z} \pi^{\A^1}_{n,n} \mathbb{S}[\eta^{-1}] \xrightarrow{\simeq} \bigoplus_{n\in \Z} \pi^{\A^1}_{n,n} KW = \underline{W}[\eta,\eta^{-1}].
\]
\end{lemma}
\begin{proof}
First we claim that the unit morphism induces an isomorphism
\[
(u_{KW})_*\colon  \bigoplus_{n\in \Z} \mathbb{S}^{n,n}[\eta^{-1}](pt) \xrightarrow{\simeq} \bigoplus_{n\in \Z} KW^{n,n}(pt).
\]
This follows from Morel's computation \cite[Corollary~6.43]{Mor12}
\[
\bigoplus_{n\in \Z} \mathbb{S}^{n,n}(pt)\cong K^{MW}_*(k).
\]
Indeed, by \cite[p.74]{Mor12} for $a\in k^*$ the quadratic form $\langle a \rangle$ corresponds under this isomorphism to $\langle a \rangle_\pi$ in the notation of \cite[Definition~6]{An12}, so the claim follows from \cite[Corollary~1]{An12}.

Using a base change argument we obtain that the unit morphism induces an isomorphism
\[
 \bigoplus_{n\in \Z} \pi^{\A^1}_{n,n} \mathbb{S}[\eta^{-1}] (F) = \bigoplus_{n\in \Z} \mathbb{S}^{n,n}[\eta^{-1}](F) \xrightarrow{\simeq} \bigoplus_{n\in \Z} KW^{n,n}(F) = \bigoplus_{n\in \Z} \pi^{\A^1}_{n,n} KW (F)
\]
for every field $F/k$. Stable $\A^1$-homotopy groups are strictly $\A^1$-invariant by \cite[Remark 5.1.13]{Mor04a}, thus \cite[Corollary 4.2.3]{Mor04b} yields the claim of the lemma.
\end{proof}

We recall from \cite[\S 5.2]{Mor04a} that the homotopy $t$-structure on $\SH(k)$ is defined by the pair of full subcategories $(\SH(k)_{\ge 0}, \SH(k)_{\le0})$ with
\begin{align*}
&\SH(k)_{\ge 0}:=\{A\in \SH(k)\ |\ \pi^{\A^1}_m(A)_*=0\text{ for all } m<0\}\\
&\SH(k)_{\le 0}:=\{A\in \SH(k)\ |\ \pi^{\A^1}_m(A)_*=0\text{ for all } m>0\}
\end{align*}
 
The full subcategories $\SH(k)_{\ge n}$, $\SH(k)_{\le n}$ are defined as
\begin{align*}
&\SH(k)_{\ge n}:=\{A\in \SH(k)\ |\ \pi^{\A^1}_m(A)_*=0\text{ for all } m<n\}=\Sigma_{S^1}^n\SH(k)_{\ge 0}\\
&\SH(k)_{\le n}:=\{A\in \SH(k)\ |\ \pi^{\A^1}_m(A)_*=0\text{ for all } m>n\}=\Sigma_{S^1}^n\SH(k)_{\le 0}.
\end{align*}
We say that an object $A$ of $\SH(k)$ is {\em $n-1$-connected} if $A$ is in $\SH(k)_{\ge n}$.

Each $X\in \SH(k)$ admits a canonical distinguished triangle
\[
X_{\ge n}\to X\to X_{\le n-1}\to X_{\ge n}[1]
\]
with $X_{\ge n}$ in $\SH(k)_{\ge n}$ and $X_{\le n-1}$ in $\SH(k)_{\le n-1}$. Sending $X$ to $X_{\ge n}$ defines a right   adjoint $\tau_{\ge n}:\SH(k)\to \SH(k)_{\ge n}$ to the inclusion $\SH(k)_{\ge n}\to \SH(k)$; similarly,  sending $X$ to $X_{\le n-1}$ defines a left adjoint $\tau_{\le n-1}:\SH(k)\to \SH(k)_{\le n-1}$ to the inclusion $\SH(k)_{\le n-1}\to \SH(k)$. Furthermore, the map $p:X_{\ge n}\to X$ is characterized up to unique isomorphism by the fact $X_{\ge n}$ is in $\SH(k)_{\ge n}$ and $p$ induces an isomorphism on homotopy sheaves $\pi^{\A^1}_{a,b}$ for $a\ge b+n$. 

The heart of the $t$-structure is denoted $\pi_*^{\A^1}(k)$; we let $EM:\pi_*^{\A^1}(k)\to \SH(k)$ denote the inclusion, and $H_0:\SH(k)\to \pi_*^{\A^1}(k)$ the truncation functor $\tau_{\le 0}\tau_{\ge0}$.  As shown in \cite[Theorem 5.2.6]{Mor04a}, 
$\pi_*^{\A^1}(k)$ is equivalent to the category of {\em homotopy modules}, where a homotopy module is a graded strictly $\A^1$ invariant Nisnevich sheaf $M_*=\oplus_nM_n$ on $\Sm/k$, together with isomorphisms $\sHom(\G_m, M_{n+1})\cong M_n$. We will henceforth identify $\pi_*^{\A^1}(k)$ with the category of homotopy modules.

\begin{definition}
Denote by  $EM(W_\Q)$ the object $EM(H_0(\mathbb{S}[\eta^{-1}]_\Q))$ of $\SH(k)$, that is,  $EM(W_\Q)$ is the  object of $\SH(k)$ associated to the homotopy module $(\underline{W}_\Q)_*=\underline{W}_\Q[\eta,\eta^{-1}]$:
\[
\pi^{\A^1}_{a,b}EM(W_\Q)=\begin{cases}\underline{W}_\Q&\text{ for }a=b\\0&\text{else.}\end{cases}
\]
\end{definition}

\begin{theorem} \label{thm:decompose_KW}
The unit morphism $u_{KW}\colon  \mathbb{S} \to KW$ induces an isomorphism
\[
u_{KW}\colon EM(W_\Q) \xrightarrow{\simeq} KW^{(0)}_\Q.
\]
\end{theorem}
\begin{proof}
Theorem~\ref{thm:KW_represents} yields 
\[
\pi^{\A^1}_{*,*} KW_\Q = \underline{W}_\Q[\eta,\eta^{-1},\beta,\beta^{-1}]=\bigoplus_{m\in \Z} \underline{W}_\Q[\eta,\eta^{-1}]\beta^m.
\]
Stable operations are homomorphisms of $\mathbb{S}$-modules, thus by Lemma~\ref{lm:unit_iso} the operation $\rho_0^{st}$ acts trivially on $\bigoplus_{n\in \Z}\pi^{\A^1}_{n,n} KW_\Q = \underline{W}_\Q[\eta,\eta^{-1}]$ and maps all the other summands to zero. Hence the unit morphism induces an isomorphism 
\[
\bigoplus_{n\in \Z} \pi^{\A^1}_{n,n} \mathbb{S}[\eta^{-1}]_\Q \xrightarrow{\simeq} \bigoplus_{n\in \Z} \pi^{\A^1}_{n,n} KW_\Q^{(0)} 
\]
while $\pi^{\A^1}_{i+n,n} KW_\Q^{(0)}=0$ for $i\neq 0$. Then, using the fact that spectrum $\mathbb{S}$ is $(-1)$ connected by Morel's connectivity theorem, we obtain that the unit morphism induces an isomorphism
\[
EM(W_\Q) =EM(H_0( \mathbb{S}[\eta^{-1}]_\Q))\cong  EM(H_0(KW_\Q^{(0)})) = KW_\Q^{(0)}.
\]
\end{proof}

\begin{corollary}\label{cor:WittDecomp}
$KW_\Q \cong \bigoplus_{m\in \Z} \Sigma_T^{4m} EM(W_\Q)$.
\end{corollary}
\begin{proof}
This follows from the theorem and Lemma~\ref{lm:iso_eigenspaces}.
\end{proof}

\begin{corollary}\label{cor:Witt}
Let $X$ be a smooth variety. Then there exists a canonical isomorphism
\[
W^{n}_\Q(X)\cong \bigoplus_{m\in \Z} H^{4m+n}_{Zar}(X,\underline{W}_\Q).
\]
\end{corollary}
\begin{proof}
This follows from the above corollary and the identification
\[
EM(W_\Q)^{i,j}(X)\cong H^{i-j}_{Nis}(X,\underline{W}_\Q)=H^{i-j}_{Zar}(X,\underline{W}_\Q).
\]
\end{proof}

\begin{remark} The last corollary could be seen as a statement about the rational degeneration of a Brown-Gersten type spectral sequence arising from the homotopy $t$-structure. The tower of spectra 
\[
\hdots \to KW_{\ge n+1}\to KW_{\ge n} \to KW_{\ge n-1} \to \hdots \to KW
\]
gives rise to a spectral sequence 
\[
E_2^{p,q}:= H^p_{Zar}(X, \pi_{-q,0}KW)\cong H^p_{Zar}(X, \underline{W}^{q}) \Longrightarrow KW^{p+q,0}(X)\cong W^{p+q}(X).
\]
Corollary~\ref{cor:WittDecomp} yields that rationally we have
\[
(KW_\Q)_{\ge n}\cong \bigoplus_{m\ge n/4} \Sigma_T^{4m} EM(W_\Q)
\]
and all the morphisms $(KW_\Q)_{\ge n+1}\to (KW_\Q)_{\ge n}$ are split injections, thus rationally the spectral sequence degenerates.
\end{remark}

\begin{remark}
Let $X$ be a smooth variety over $\mathbb{R}$, then \cite[Theorem~1.1]{Jac15} gives an isomorphism 
\[
H^q_{Zar}(X, \underline{W}_\Q)= H^q_{Zar}(X, \mathcal{I}^n_\Q)\cong H^q_{sing}(X(\mathbb{R}), \Q).
\]
Hence Corollary~\ref{cor:Witt} together with the rational degeneration of Atiyah-Hirzebruch spectral sequence for (topological) $K$-theory of real vector bundles identifies
\[
W^{n}_\Q(X)\cong KO_{top}^n(X(\mathbb{R}))\otimes \Q.
\]
In fact, using real algebraic geometry techniques one can show \cite{Br84} that it is sufficient to invert $2$ to obtain a similar isomorphism:
\[
W^{n}(X)\otimes \Z[\tfrac{1}{2}] \cong KO_{top}^n(X(\mathbb{R}))\otimes \Z[\tfrac{1}{2}].
\]
\end{remark}

\begin{corollary} \label{cor:no_extensions}
The canonical morphism $u\colon \mathbb{S}\to EM(W_\Q)$ induces isomorphisms
\begin{enumerate}
\item
$\mathcal{H}om_{\SH(k)}(EM(W_\Q), EM(W_\Q))\xrightarrow{\simeq} \mathcal{H}om_{\SH(k)}(\mathbb{S}_\Q, EM(W_\Q))=EM(W_\Q)$,
\item
$u\wedge \id \colon EM(W_\Q)=\mathbb{S}\wedge EM(W_\Q)\xrightarrow{\simeq} EM(W_\Q)\wedge EM(W_\Q)$.
\end{enumerate}
Here $\mathcal{H}om$ denotes the internal Hom.
\end{corollary}
\begin{proof}
Using Theorem~\ref{thm:decompose_KW} we identify $EM(W_\Q)=KW^{(0)}_\Q$. It is sufficient to show that the unit morphism $u\colon \mathbb{S}\to KW^{(0)}_\Q$ induces isomorphisms on the sheaves of stable $\A^1$-homotopy groups of respective spectra. The base change argument combined with the theory of strictly $\A^1$-invariant sheaves in the same way as in the proof of Lemma~\ref{lm:unit_iso} yields that it is sufficient to show that the unit morphism induces isomorphisms
\begin{align*}
&u^{KW^{(0)}_\Q}\colon (KW^{(0)}_\Q)^{*,*}(KW^{(0)}_\Q)\xrightarrow{\simeq} (KW^{(0)}_\Q)^{*,*}(pt),\\
&u_{KW^{(0)}_\Q}\colon  (KW^{(0)}_\Q)_{*,*}(pt)\xrightarrow{\simeq} (KW^{(0)}_\Q)_{*,*}(KW^{(0)}_\Q).
\end{align*}

Note that for every $m$ morphism $ \Sigma^{4m}_T \beta^m\colon \Sigma^{4m}_T \mathbb{S}\to KW_\Q$ factors in a natural way through $KW^{(m)}_\Q$, hence for the morphism $\mathcal{B}$ given in Definition~\ref{def:B} we have
\[
\mathcal{B}^{KW_\Q}=\prod_{m\in \Z}  (\Sigma^{4m}_T\beta^m)^{KW_\Q} \colon \prod_{m\in \Z} KW^{*,*}_Q(KW^{(m)}_\Q) \to \prod_{m\in \Z} KW^{*,*}_Q(\Sigma^{4m}_T \mathbb{S} ).
\]
This homomorphism is an isomorphism by Theorem~\ref{thm:rational_stable_operations_KW}, thus, in particular, the unit morphism induces an isomorphism 
\[
u^{KW_\Q}\colon KW_\Q^{*,*}(KW^{(0)}_\Q)\xrightarrow{\simeq} KW_\Q^{*,*}(pt).
\]
The claim for operations follows, since $u^{KW^{(0)}_\Q}$ is a retract of isomorphism $u^{KW_\Q}$.

The case of cooperations is similar: by Theorem~\ref{thm:cooperations} we have an isomorphism
\[
\mathcal{B}_{KW_\Q}= \bigoplus_{m\in \Z}  (\Sigma^{4m}_T\beta^m)_{KW_\Q}  \colon \bigoplus_{m\in \Z} (KW_\Q)_{*,*}( \Sigma^{4m}_T \mathbb{S}) \xrightarrow{\simeq} \bigoplus_{m\in \Z} (KW_\Q)_{*,*}( KW^{(m)}_\Q)
\]
which yields an isomorphism
\[
u_{KW_\Q}\colon (KW_\Q)_{*,*}(pt) \xrightarrow{\simeq} (KW_\Q)_{*,*}(KW^{(0)}_\Q).
\]
The homomorphism $u_{KW_\Q^{(0)}}$ is an isomorphism being a direct summand of an isomorphism $u_{KW_\Q}$.
\end{proof}

\section{Homotopy sheaves of $\mathbb{S}[\eta^{-1}]_\Q$}

Let  $\Spt_T(k)$ denote the model category of symmetric $T$-spectra over $k$, with its motivic model structure (see e.g. \cite{Jardine00}). As discussed in \cite[\S3]{Lev13b}, one can realize the category  $ \SH(k)_{\ge n}$ as the homotopy category of a suitable right Bousfield localization of  $\Spt_T(k)$. More specifically, let $K_n$ be the set of objects of $\Spt_T(k)$,
\[
K_n:=\{F_q(\Sigma^p_{S^1}X_+)\ |\ X\in \Sm/k, p\ge n+q \}.
\]
Here $F_q:\Spc_\bullet(k)\to \Spt_T(k)$ is the left adjoint to the $q$th evaluation functor $\ev_q:\Spt_T(k)\to \Spc_\bullet(k)$, $\ev_q((A_0, A_1,\ldots))=A_q$, and $K_n$ is denoted $K_{n,-\infty}$ in \cite[\hbox{\it loc. cit.}]{Lev13b}. Let $\Spt_T(k)(K_n)$ denote the right Bousfield localization of $\Spt_T(k)$ with respect to the set of $K_n$-local maps in $\Spt_T(k)$. The general theory of Bousfield localization tells us that $\Hho\Spt_T(k)(K_n)$ is equivalent to the essentially  full subcategory of $\SH(k)$ consisting of the  $K_n$-colocal objects of $\Spt_T(k)$; we henceforth consider  $\Hho\Spt_T(k)(K_n)$ as this essentially full subcategory of $\SH(k)$.

\begin{lemma} \label{lem:Bousfield} $\Hho\Spt_T(k)(K_n)=\SH(k)_{\ge n}$.
\end{lemma}

\begin{proof}Let $r_n^s:\SH(k)\to \Hho\Spt_T(k)(K_n)$ denote the right adjoint to the inclusion functor $i_n^s:\Hho\Spt_T(k)(K_n)\to \SH(k)$ given by the general theory of right Bousfield localization, and let $f_n^s=i_n^s\circ r_n^s$. By \cite[lemma 4.5]{Lev13b}, $f_n^sX$ is in $\SH(k)_{\ge n}$ for all $X\in \SH(k)$ and thus $\Hho\Spt_T(k)(K_n)\subset \SH(k)_{\ge n}$.  Thus, by the universal property of $\tau_{\ge n}$, we  have for $X\in \SH(k)$  a commutative triangle
\[
\xymatrix{
f_n^sX\ar[r]\ar[d]&X\\
\tau_{\ge n}X\ar[ur]
}
\]
By \cite[lemma 4.3]{Lev13b}, the map $f_n^sX\to X$ induces an isomorphism on $\pi^{\A^1}_{a,b}$ for all $a\ge b+n$, and thus $f_n^sX\to \tau_{\ge n}X$ is an isomorphism. Thus $\SH(k)_{\ge n}\subset \Hho\Spt_T(k)(K_n)$.
\end{proof}

\begin{lemma}\label{lem:smash}
For every $m,n\in \Z$ we have $ \SH(k)_{\ge m}\wedge \SH(k)_{\ge n}  \subset  \SH(k)_{\ge m+n}$.
\end{lemma}
\begin{proof}   By Lemma~\ref{lem:Bousfield} we may show instead that  
\[
\Hho\Spt_T(k)(K_m)\wedge \Hho\Spt_T(k)(K_n)\subset 
\Hho\Spt_T(k)(K_{m+n}). 
\]
By \cite[theorem 2.4]{Lev13b} $\Hho\Spt_T(k)(K_m)$ is the full subcategory of $\SH(k)$ of $K_m$-cellular objects. 

We recall from \cite[\S 4.3]{Jardine00} that  the smash product of symmetric spectra defines a closed symmetric monoidal structure on $\Spt_T(k)$, that is, there is an internal Hom functor $\sHom_{\SH(k)}(-,-)$ and natural isomorphism
\[
\Hom_{\SH(k)}(A\wedge B, C)\cong \Hom_{\SH(k)}(A, \sHom_{\SH(k)}(B, C)).
\]
Thus,  $-\wedge B$ preserves  retracts and  colimits for all $B\in \Spt_T(k)$ and it therefore suffices to show that  $K_m\wedge K_n\subset K_{m+n}$.  

For this,  we recall from \cite[Corollary 4.18]{Jardine00} the existence of a natural isomorphism  $F_a\sX\wedge F_b\sY\cong F_{a+b}(\sX\wedge \sY)$ for $\sX, \sY$ in $\Spc_\bullet(k)$. Thus, for $X, Y\in \Sm/k$, we have 
\[
F_q\Sigma^p_{S^1}X_+\wedge F_{q'}\Sigma^{p'}_{S^1}Y_+\cong
F_{q+q'}\Sigma^{p+p'}_{S^1} (X\times Y)_+,
\]
so $K_m\wedge K_n\subset K_{m+n}$, as desired.
\end{proof}

\begin{corollary}\label{cor:smash_tstructure}
For $A\in \SH(k)_{\ge 1}$ we have $A\wedge A \cong A$ if and only if $A\cong 0$.
\end{corollary}
\begin{proof} Suppose that $A\in \SH(k)_{\ge 1}$ and  $A\wedge A \cong A$. The homotopy $t$-structure is non-degenerate, as $\pi^{\A^1}_{a,b}A=0$ for all $a,b$ if and only if $A\cong 0$. Thus, if $A\not\cong0$, there is a maximal $n\ge 1$ so that $A\in \SH(k)_{\ge n}$, in other words, $\pi^{\A^1}_{n+b,b}A\neq 0$ for some $b\in \Z$. But  by Lemma~\ref{lem:smash},  $A\wedge A$ is in $\SH(k)_{\ge 2n}$ and thus  $\pi^{\A^1}_{n+b,b}A\wedge A=0$ for all $b\in \Z$, contrary to the assumption $A\wedge A\cong A$. Thus, we must have $A\cong 0$.
\end{proof}

\begin{theorem}\label{thm:Main}
The canonical morphism $\mathbb{S}[\eta^{-1}]_\Q\to EM(W_\Q)$ is an isomorphism.
\end{theorem}
\begin{proof}
Recall that by Morel's connectivity theorem $\mathbb{S}$ is $(-1)$-connected, thus there exists a canonical triangle
\[
A\to \mathbb{S}[\eta^{-1}]_\Q\to EM(W_\Q) \to A[1]
\]
with $A\in \SH(k)_{\ge 1}$. We clearly have
\[
EM(W_\Q) =\mathbb{S}\wedge EM(W_\Q)=\mathbb{S}[\eta^{-1}]_\Q\wedge EM(W_\Q),
\]
thus smashing the triangle with $EM(W_\Q)$ and using Corollary~\ref{cor:no_extensions} we obtain
\[
A\wedge EM(W_\Q)=0.
\]
Smashing the triangle with $A$ we see that $A\wedge A\cong A$. Corollary~\ref{cor:smash_tstructure} yields $A=0$ and the claim follows.
\end{proof}

\begin{corollary} 
$KW_\Q \cong \bigoplus_{m\in \Z} \Sigma_T^{4m} \mathbb{S}[\eta^{-1}]_\Q$.
\end{corollary}

\begin{proof} This follows from Corollary~\ref{cor:WittDecomp} and Theorem~\ref{thm:Main}.
\end{proof}

\begin{remark} The symmetry involution $\P^1\wedge\P^1\to \P^1\wedge \P^1$ induces an involution $\tau$ of the sphere spectrum, and thereby a decomposition in $\SH(k)[1/2]$:
\[
 \mathbb{S}[1/2]= \mathbb{S}^+\oplus \mathbb{S}^-.
 \]
 Morel has givenan explicit computation of $\tau$ in terms of $\eta$, which shows that $\eta$ is zero on $ \mathbb{S}^+$ and is an isomorphism on $\mathbb{S}^-$. In particular, 
 \[
 \mathbb{S}^-\cong  \mathbb{S}[\eta^{-1},1/2].
 \]
We can rephrase Theorem~\ref{thm:Main} as saying that the unit map $\mathbb{S}_\Q\to EM(W_\Q)$ induces an isomorphism   $\mathbb{S}^-_\Q\to EM(W_\Q)$.

As $\mathbb{S}$ is the unit for the monoidal structure on $\SH(k)$, the decomposition $ \mathbb{S}[1/2]= \mathbb{S}^+\oplus \mathbb{S}^-$ induces a decomposition of $\SH(k)[1/2]$ as a product of symmetric monoidal triangulated categories
\[
\SH(k)[1/2]=\SH(k)^+\times \SH(k)^-,
\]
with $ \mathbb{S}^+$ the unit for $\SH(k)^+$ and $ \mathbb{S}^-$ the unit for $\SH(k)^-$. Combining a theorem of R\"ondigs-{\O}stv{\ae}r \cite[Theorem 1.1]{RO}, with a result of Cisinsky-D\'eglise \cite[Theorem 16.2.13]{CD} shows that the unit map to the motivic cohomology spectrum $H\Z$ induces an isomorphism $\mathbb{S}^+_\Q\to H\Q$, which in turn defines an equivalence of $\SH(k)^+_\Q$ with the homotopy category of $H\Q$-modules. The theorem of R\"ondigs-{\O}stv{\ae}r shows that this category is equivalent to Voevodsky's triangulated category of motives with $\Q$-coefficients, $DM(k)_\Q$.

In the next section we will use Theorem~\ref{thm:Main}  to give an analogous description of   $\SH(k)_\Q^-$; see Theorem~\ref{thm:Equiv}  and Corollary~\ref{cor:Equiv}.
\end{remark}

\begin{theorem}[motivic Serre finiteness]\label{thm:SFinite} Let $k$ be a field. The  sheaf $\pi^{\A^1}_{a,b}(\mathbb{S}^-_k)$ is torsion for $a\neq b$, i.e., $\pi^{\A^1}_n(\mathbb{S}^-_k)_*$ is torsion for $n\neq0$. The
 sheaf $\pi^{\A^1}_n(\mathbb{S}_k)_q$ is torsion for  $n\neq0$ and $q\ge0$.
\end{theorem}

\begin{proof}  The first assertion follows directly from Theorem~\ref{thm:Main}: $\mathbb{S}^-_{k\Q}\cong EM(W_\Q)$ implies that $\mathbb{S}^-_{k\Q}$ is in the heart of the $t$-structure and therefore $\pi^{\A^1}_{a,b}(\mathbb{S}^-_k)_\Q=0$ for $a\neq b$. 

Using the decomposition $\mathbb{S}_{k\Q}=\mathbb{S}_{k\Q}^+\oplus \mathbb{S}_{k\Q}^-$ and the theorem of Cisinski-D\'eglise gives us 
\[
\pi^{\A^1}_{a,b}(\mathbb{S}_k)(F)_\Q=H^{-a}(F,\Q(-b))\oplus \pi^{\A^1}_{a,b}(\mathbb{S}^-_k)(F)_\Q.
\]
For $a\neq b$, only the term $H^{-a}(F,\Q(-b))$ remains; this is zero for $b>0$ by the vanishing of motivic cohomology in negative weights, or for $b=0$, $a\neq0$. 
\end{proof}

\begin{remark} For $0>a>b$, the argument used above shows that 
\[
\pi^{\A^1}_{a,b}(\mathbb{S}_k)_\Q(F)=H^{-a}(F,\Q(-b))
\]
which is in general not zero: for instance, if $K$ is a number field, $H^1(K,\Q(n))$ is a $\Q$-vector space of dimension $r_2$ (the number of conjugate pairs of complex embeddings of $K$) if $n$ is even and of dimension $r_1+r_2$ ($r_1$ is the number of real embeddings of $K$) for $n$ odd. For $F=K(t_1,\ldots, t_r)$ a pure transcendental extension of $K$, the cup product with $\{t_1,\ldots, t_a\}\in H^a(F,\Q(a))$ is an injective map $H^1(K,\Q(n))\to H^{a+1}(F, \Q(a+n))$, for all $a=1,\ldots, r$. As the map $H^{-a}(F,\Q(-b))\to H^{-a}(F',\Q(-b))$ is injective for each field extension $F\to F'$,  the sheaf $\pi^{\A^1}_{a,b}(\mathbb{S}_k)_\Q$ is non-zero for $0>a>b$ and $k$ a characteristic zero field.  

On the other hand, by Morel's $\A^1$-connectedness theorem, $\mathbb{S}_k$ is in $\SH(k)_{\ge0}$, so $\pi^{\A^1}_{a,b}(\mathbb{S}_k)=0$ for $a<b$. 
\end{remark}

\section{The category $\SH(k)^-_\Q$}\label{sec:WittMot}

We recall briefly the construction of  the $\A^1$-derived category $D_{\A^1}(k,\Z)$.  We have the category of complexes of  abelian groups on $\Sm/k$, $C(Sh^{Nis}_{Ab}(\Sm/k))$. For $X\in \Sm/k$, we have the free abelian sheaf $\Z(X)$ generated by $X$, that is the sheafification of the presheaf
\[
Y\mapsto \oplus_{s\in \Hom_{\Sm/k}(Y,X)}\Z
\]
We write $D^{n+1}(X)$ for  $\text{cone}(id:\Z(X)[n]\to \Z(X)[n])$ with canonical map $i_n:\Z(X)[n]\to D^{n+1}(X)$.

We give the category of complexes of Nisnevich sheaves of abelian groups $C(Sh_{Ab}^{Nis}(\Sm/k))$ the  model structure with weak equivalences the Nisnevich quasi-isomorphisms and generating set of cofibrations the maps $i_n:\Z(X)[n]\to D^{n+1}(X)$, $0\to \Z(X)[n]$,  $X\in \Sm/k$, $n\in \Z$.    Let $C_{\A^1}(Sh_{Ab}^{Nis}(\Sm/k))$ be the Bousfield localization with respect to the projection maps $\Z(X\times\A^1)[n]\to \Z(X)[n]$, $X\in \Sm/k$. An object $A$ of the derived category $D(Sh_{Ab}^{Nis}(\Sm/k))=\Hho C(Sh_{Ab}^{Nis}(\Sm/k))$ is called $\A^1$-local if 
\[
\Hom_{D(Sh_{Ab}^{Nis}(\Sm/k))}(C, A)\xrightarrow{p^*}\Hom_{D(Sh_{Ab}^{Nis}(\Sm/k))}(C\otimes\Z(\A^1), A)
\]
is an isomorphism for all $C$. We let $D_{\A^1}^{eff}(k)\subset D(Sh_{Ab}^{Nis}(\Sm/k))$ be the full subcategory of $\A^1$-local objects. The general theory of Bousfield localization tells us that the localization map $D(Sh_{Ab}^{Nis}(\Sm/k))\to \Hho C_{\A^1}(Sh_{Ab}^{Nis}(\Sm/k))$ admits a left adjoint 
\[
L_{\A^1}:\Hho C_{\A^1}(Sh_{Ab}^{Nis}(\Sm/k))\to D(Sh_{Ab}^{Nis}(\Sm/k))
\]
with essential image equal to $D_{\A^1}^{eff}(k,\Z)$ and induces an equivalence of triangulated categories $:\Hho C_{\A^1}(Sh_{Ab}^{Nis}(\Sm/k))\sim D_{\A^1}^{eff}(k)$.

One then forms the category of $\G_m$-spectrum objects in $C_{\A^1}(Sh_{Ab}^{Nis}(\Sm/k))$, following Hovey \cite{Hovey} and Jardine \cite{Jardine00}: objects are sequences $\sC:=(C_0, C_1,\ldots)$ with $C_n$ in $C_{\A^1}(Sh_{Ab}^{Nis}(\Sm/k))$ together with bonding morphisms $\epsilon_n:\Z(\G_m)\otimes C_n\to C_{n+1}$. This carries a model structure (the stable model structure) and $D_{\A^1}(k,\Z)$ is the associated homotopy category.

There is also a version of a stable model category of $\G_m$-symmetric spectrum objects in $C_{\A^1}(Sh_{Ab}^{Nis}(\Sm/k))$, which carries a closed symmetric monoidal structure and has equivalent homotopy category, making $D_{\A^1}(k,\Z)$ a tensor triangulated category.

Replacing sheaves of abelian groups with sheaves of $\Q$-vector spaces defines the triangulated tensor categories $D_{\A^1}^{eff}(k,\Q)$ and $D_{\A^1}(k,\Q)$

As in \cite[\S 5.3.35]{CD}, the  Quillen equivalence of the $\Q$-localization of the  category of classical spectra with the model category of complexes of $\Q$-vector spaces induces an equivalence of $\SH(k)_\Q$ with the rational $\A^1$-derived category $D_{\A^1}(k,\Q)$. The equivalence arises via a Quillen pair of functors of model categories, giving a pair of exact adjoint functors of triangulated categories
\[
N:\SH(k)\xymatrixrowsep{2pt}\xymatrix{\ar@<3pt>[r]&\ar@<3pt>[l]}D_{\A^1}(k,\Z):EM
\]

This Quillen pair induces a Quillen equivalence of the rational localizations of these model categories, giving us the adjoint pair of equivalences of triangulated tensor categories
\[
N:\SH(k)_\Q\xymatrixrowsep{2pt}\xymatrix{\ar@<3pt>[r]&\ar@<3pt>[l]}D_{\A^1}(k,\Q):EM
\]
The decomposition of the unit $\mS_k[1/2]$ into $\mS_k^+$ and $\mS_k^-$ gives the corresponding decomposition of $\SH(k)_\Q$ and $D_{\A^1}(k,\Q)$, 
\[
\SH(k)_\Q=\SH(k)^+_\Q\times\SH(k)_\Q^-;\quad D_{\A^1}(k,\Q)=D_{\A^1}(k,\Q)^+\times D_{\A^1}(k,\Q)^-
\]
compatible with $N$ and $EM$. In particular, we have the adjoint pair of equivalences of triangulated tensor categories
\begin{equation}\label{eqn:Equiv1}
N:\SH(k)^-_\Q\xymatrixrowsep{2pt}\xymatrix{\ar@<3pt>[r]&\ar@<3pt>[l]}D_{\A^1}(k,\Q)^-:EM
\end{equation}

We proceed to construct a category of {\em Witt motives over $k$}.

Consider the sheaf of  Witt rings $\underline{W}$ on $\Sm/k$. The identification $\underline{W}\cong \pi_{n,n}(\mS_k)$ (for all $n>0$) gives us the canonical isomorphism of $W$-modules $\epsilon_W:\underline{W}\to \sHom(\G_m, \underline{W})$. If $M$ is a sheaf of $\underline{W}$-modules, we have the canonical morphism
$\epsilon_M:M\to \sHom(\G_m, M)$ of $\underline{W}$-modules defined as the composition
\[
M\xrightarrow{\alpha} \sHom_{\underline{W}}(\underline{W}, M)\xrightarrow{\epsilon_W(1)^*}\sHom(\G_m, M)
\]
where $\alpha$ is the canonical map, described on sections as $m\mapsto f_m$, $f_m(w)=wm$, and $\epsilon_W(1):\G_m\to \underline{W}$ is the evalution of $\epsilon_W$ on the unit section of $W$. Thus, each strictly $\A^1$-invariant $\underline{W}$-module $M$ gives us an object 
\[
M_*:=((M, M,\ldots), \epsilon_M:M\to \sHom(\G_m,M)) 
\]
of $D_{\A^1}(k,\Z)$, which one easily sees is in the minus summand $D_{\A^1}(k,\Z)^-$ (after inverting 2). 

Explicitly, if we take the model $\underline{W}=\underline{K}_0^{W}:=\underline{K}_0^{MW}/(h)$, the map $\epsilon_W$ sends a section $x\in \underline{W}(U)$ to the section $ p_1^*(\eta\cdot[t])\cdot p_2^*x\in \underline{W}(\G_m\wedge U_+)$, where $t$ is the canonical unit on $\G_m$ and $[t]\in \underline{K}^{MW}_1(\G_m)$ the corresponding section. The map $\epsilon_M$ has a similar expression. Alternatively, one can simply compose $x$ with the stable Hopf map $\eta:\G_m\wedge \mS\to \sS$. The fact that $x\circ\eta=\eta\cdot p_1^*[t]\cdot p_2^*x$ follows from the commutativity of the diagram
\[
\xymatrix{
\G_m\wedge\mS\ar[r]^{t}\ar[d]_\eta&\G_m\wedge\mS\ar[d]^\eta\\
\mS\ar[r]_{\text{id}}&\mS,}
\]
noting that the canonical unit on $\G_m$ is just the identity map $\text{id}_{\G_m}$. This gives an explicit description of the map $\epsilon_M:M\to \sHom(\G_m, M)$ as $\epsilon_M(x)=p_1^*(\eta\cdot [t])\cdot p_2^*x$ for $x$ a section of $M$; the alternate description in terms of composition with the algebraic Hopf map is only available for $M$ satisfying a $\G_m$-stability condition.

We have the category of complexes of  $\underline{W}$-modules on $\Sm/k$, $C(Sh^{Nis}_{\underline{W}}(\Sm/k))$. For $X\in \Sm/k$, we have the free $\underline{W}$-module $\underline{W}(X):=\underline{W}\otimes_\Z\Z(X)$ generated by $X$, 
and let $D^{n+1}_W(X):=\underline{W}\otimes_\Z D^{n+1}(X)$. 

We give  $C(Sh^{Nis}_{\underline{W}}(\Sm/k))$ the (combinatorial, proper) model structure for which the weak equivalences are the (Nisnevich) quasi-isomorphisms and generating cofibrations   $i_n:\underline{W}(X)[n]\to D^{n+1}_W(X)$, $0\to \underline{W}(X)[n]$, $X\in \Sm/k$, $n\in \Z$. The derived category  of Nisnevich sheaves of $\underline{W}$-modules on $\Sm/k$, $D(Sh^{Nis}_{\underline{W}}(\Sm/k))$, is the homotopy category of $C(Sh^{Nis}_{\underline{W}}(\Sm/k))$. The construction outlined above gives us for each $C$ in $C(Sh^{Nis}_{\underline{W}}(\Sm/k))$  a canonical morphism $\epsilon_C:C\to \sHom(\G_m, C)$ in $C(Sh^{Nis}_{\underline{W}}(\Sm/k))$.

We call an object $C$ of $C(Sh^{Nis}_{\underline{W}}(\Sm/k))$ $\A^1$-local if $C$ is fibrant and the pull-back $C\to\sHom(\A^1, C)$  is a quasi-isomrophism. We call $C$ a $\G_m$-infinite loop complex if $\epsilon_C:C\to \sHom(\G_m, C)$ is a quasi-isomorphism. 

We denote by $DM_W(k)$ the full subcategory of $D(Sh^{Nis}_{\underline{W}}(\Sm/k))$ with objects the complexes that are fibrant,  $\A^1$-local and $\G_m$-infinite loop complexes. 
 
 \begin{lemma} $DM_W(k)$ is equivalent to the Verdier localization of $D(Sh^{Nis}_{\underline{W}}(\Sm/k))$ with respect to the localizing subcategory $\sC$ generated by the following two types of objects:
 \begin{enumerate}
 \item Cone$(\text{const}_C:C\to \sHom(\A^1, C))$.
 \item Cone$(\epsilon_C:C\to \sHom(\G_m, C))$
 \end{enumerate}
 More precisely, the inclusion $DM_W(k)\to D(Sh^{Nis}_{\underline{W}}(\Sm/k))$  admits a left adjoint $L_{\A^1, W}:
 D(Sh^{Nis}_{\underline{W}}(\Sm/k))\to DM_W(k)$ which sends the objects of type (1) and (2) to zero objects, and the resulting exact functor $D(Sh^{Nis}_{\underline{W}}(\Sm/k))/\sC\to DM_W(k)$ is an equivalence.
 \end{lemma}
 
 \begin{proof} This follows by forming the left Bousfield localization $C_{\A^1,\G_m}(Sh^{Nis}_{\underline{W}}(\Sm/k))$ of $C(Sh^{Nis}_{\underline{W}}(\Sm/k))$ with respect to the set of maps $\text{const}_C:C\to \sHom(\A^1, C)$ and $\epsilon_C:C\to \sHom(\G_m, C)$, as $C$ runs over a set of fibrant models of generators for $C(Sh^{Nis}_{\underline{W}}(\Sm/k))$. This localization exists because $C(Sh^{Nis}_{\underline{W}}(\Sm/k))$  is a proper combinatorial model category.
 \end{proof}
 
We give  $DM_W(k)$ the tensor structure induced from $D(Sh^{Nis}_{\underline{W}}(\Sm/k))$ by the localization $D(Sh^{Nis}_{\underline{W}}(\Sm/k))\to D(Sh^{Nis}_{\underline{W}}(\Sm/k))/\sC$. We call $DM_W(k)$ the triangulated tensor category of {\em Witt motives} over $k$. 
 
Sending $C$ in $C(Sh^{Nis}_{\underline{W}}(\Sm/k))$ to the object $C_*:=((C, C, \ldots), \epsilon_C:C\to \sHom(\G_m, C))$ factors through the localization and thus gives a well-defined exact functor 
\[
Mot:DM_W(k)\to D_{\A^1}(k,\Z).
\]
After inverting 2, the image of $Mot$ lies in $D_{\A^1}(k,\Z)^-$. One can also compose $Mot$ with the Eilenbeg-MacLane functor to give the exact functor
\[
EM\circ Mot:DM_W(k)\to \SH(k).
\]

Replacing $\underline{W}$ with $\underline{W}_\Q$  gives us the full triangulated sub-category $DM_W(k)_\Q$ of $D(Sh^{Nis}_{\underline{W}_\Q}(\Sm/k))$ and exact functor
\[
Mot:DM_W(k)_\Q\to D_{\A^1}(k,\Q)^-.
\]

\begin{theorem} \label{thm:Equiv} The functor $Mot:DM_W(k)_\Q\to D_{\A^1}(k,\Q)^-$ is an equivalence of triangulated tensor categories. Combined with the equivalence \eqref{eqn:Equiv1} gives us an   equivalence of triangulated tensor categories
$DM_W(k)_\Q\cong \SH(k)^-_\Q$.
\end{theorem}

The proof is accomplished by first identifying $D_{\A^1}(k,\Q)^-$ with the homotopy category of $N(\mS_k)[\eta^{-1}]_\Q$-modules in $D_{\A^1}(k,\Q)$ and then showing that this category is equivalent to $DM_W(k)_\Q$.

$N(\mS_k)_\Q$ is the unit for the symmetric monoidal struture in $D_{\A^1}(k,\Q)$, making $N(\mS_k)[\eta^{-1}]_\Q$ into a monoid object  in $D_{\A^1}(k,\Q)^-$, in fact, the unit in this symmetric monoidal category. Let $\underline{W}_{\Q*}$ denote the $\G_m$-spectrum object which is $\underline{W}_{\Q}$ in each degree, with bonding maps given by the maps $\epsilon_W$.  Theorem~\ref{thm:Main} and adjunction tells us that the unit map $N(\mS_k)[\eta^{-1}]_\Q\to \underline{W}_{\Q*}$ is a weak equivalence, hence $\underline{W}_{\Q*}$ is also a unit in $D_{\A^1}(k,\Q)^-$.

As this latter category is the homotopy category of a $\Q$-linear additive symmetric monoidal combinatorial tractable\footnote{see \cite[Definition 3.1.27]{CD}} model category satisfying the monoid axiom, we may use \cite[Theorem 7.1.8]{CD} to form a cofibrant replacement $\underline{W}^{cof}_{\Q*}\to \underline{W}_{\Q*}$  of the commutative monoid $\underline{W}_{\Q*}$ in the model category of commutative monoids in the category $C_{\A^1}(Sh_{\Q\text{-Vec}}^{Nis}(\Sm/k))^\Sigma_{\G_m}$ of  $\G_m$-symmetric spectra in $C_{\A^1}(Sh_{\Q\text{-Vec}}^{Nis}(\Sm/k))$. We have the model category $\underline{W}^{cof}_{\Q*}\text{-Mod}$ of  $\underline{W}^{cof}_{\Q*}$-modules in $C_{\A^1}(Sh_{\Q\text{-Vec}}^{Nis}(\Sm/k))^\Sigma_{\G_m}$, objects being symmetric spectra $\sE=((\sE_0, \sE_1,\ldots), \epsilon^\sE_*)$ together with a $\underline{W}^{cof}_{\Q*}$-module structure
\[
\mu_\sE:\underline{W}^{cof}_{\Q*}\wedge \sE\to \sE
\]

As $\underline{W}_{\Q*}$ is also a commutative monoid in $C_{\A^1}(Sh_{\Q\text{-Vec}}^{Nis}(\Sm/k))^\Sigma_{\G_m}$, we have 
the category $\underline{W}_{\Q*}\text{-Mod}$; the morphism $\underline{W}^{cof}_{\Q*}\to \underline{W}_{\Q*}$ of commutative monoids induces the functor
\[
\underline{W}_{\Q*}\wedge_{\underline{W}^{cof}_{\Q*}}-:\underline{W}^{cof}_{\Q*}\text{-Mod}\to \underline{W}_{\Q*}\text{-Mod}
\]
By \cite[Theorem 4.3]{SS}, this is a left Quillen equivalence of model categories,  so  $\Hho \underline{W}_{\Q*}\text{-Mod}$ is equivalent to  $\Hho\underline{W}^{cof}_{\Q*}\text{-Mod}$.

The functor
\[
\underline{W}^{cof}_{\Q*}\wedge-:C_{\A^1}(Sh_{\Q\text{-Vec}}^{Nis}(\Sm/k))^\Sigma_{\G_m}\to
\underline{W}^{cof}_{\Q*}\text{-Mod}
\]
is a left Quillen functor with forgetful functor as right adjoint. This gives us  the pair of adjoint functors
\[
\underline{W}^{cof}_{\Q*}\wedge-:D_{\A^1}(k,\Q)^-\xymatrixrowsep{2pt}\xymatrix{\ar@<3pt>[r]&\ar@<3pt>[l]} 
 \Hho \underline{W}^{cof}_{\Q*}\text{-Mod}:\text{forget}
 \]
Since  $\underline{W}^{cof}_{\Q*}$ is isomorphic to the unit $N(\mS_k)[\eta^{-1}]_\Q$ in $D_{\A^1}(k,\Q)^-$, the unit and co-unit maps of the adjunction are natural isomorphisms, so these functors are equivalences. As the free $\underline{W}^{cof}_{\Q*}$-module functor is a monoidal functor, the two categories are equivalent as tensor  triangulated categories.

To complete the proof of Theorem~\ref{thm:Equiv} it suffices to define an equivalence of $DM_W(k)$ with $\Hho \underline{W}_*\text{-Mod}$ as tensor triangulated categories. We have the 0-complex functor
\[
\Omega^\infty_{\G_m}:  \underline{W}_*\text{-Mod}\to  C_{\A^1,\G_m}(Sh^{Nis}_{\underline{W}}(\Sm/k))
\]
which sends a $\G_m$-spectrum object $\sE=(\sE_0, \sE_1,\ldots)$ to $\sE_0$; the fact that $\sE$ is a $\underline{W}_{*}$-module implies that $\sE_0$ is a complex of $\underline{W}$-modules. $\Omega^\infty_{\G_m}$ is a right Quillen functor, with left adjoint the $\G_m$-infinite suspension functor
\[
\Sigma^\infty_{\G_m}:  C_{\A^1,\G_m}(Sh^{Nis}_{\underline{W}}(\Sm/k))\to  \underline{W}_{*}\text{-Mod}
\]
Using the maps $\epsilon_M$ to define bonding maps, we have as well the constant sequence functor
\[
\text{const}: C_{\A^1,\G_m}(Sh^{Nis}_{\underline{W}}(\Sm/k))\to  \underline{W}_{*}\text{-Mod}.
\]
and the $\epsilon_M$ define as well a natural transformation $\epsilon_W:\Sigma^\infty_{\G_m}\to \text{const}$.

For $C\in C_{\A^1,\G_m}(Sh^{Nis}_{\underline{W}}(\Sm/k))$ fibrant, the transformation $\epsilon_W(C):\Sigma^\infty_{\G_m}C\to \text{const}C$ is a stable weak equivalence, which shows that the unit of the adjunction on the associated homotopy categories is an isomorphism. 

\begin{lemma} Let $\sE=(\sE_0, \sE_1,\ldots)$ be a fibrant object in $\underline{W}_{*}\text{-Mod}$. Then for each $n$ the map $\epsilon_{\sE_n}:\sE_n\to \sHom(\G_m, \sE_n)$ is a weak equivalence.
\end{lemma}

\begin{proof} The map $\eta[t]:W\to \Omega_{\G_m}W$ is an isomorphism. Using this map in each level defines an isomorphism of $ \underline{W}_*$ modules
\[
\eta[t]: \underline{W}_*\to \Omega_{\G_m} \underline{W}_*
\]
Let $\epsilon_{\sE_*}:\sE\to \Omega_{\G_m}\sE$ be the map defined by the collection of morphisms $\epsilon_{\sE_n}:\sE_n\to \Omega_{\G_m}\sE_n$. We have the commutative diagram iin $\Ho  \underline{W}_*\text{-Mod}$
\[
\xymatrix{
\underline{W}_*\otimes_{\underline{W}_*}\sE\ar[r]^-m\ar[d]_{\eta[t]}&\sE\ar[d]^{\epsilon_{\sE_*}}\\
\Omega_{\G_m}\underline{W}_*\otimes_{\underline{W}_*}\sE\ar[r]_-m&\Omega_{\G_m}\sE}
\]
As the two horizontal arrows and the left-hand vertical arrow are isomorphisms, so is $\epsilon_{\sE_*}$. As $\sE$ is fibrant, $\sE$ is an $\Omega_{\G_m}$-spectrum, so the individual maps $\epsilon_{\sE_n}:\sE_n\to \Omega_{\G_m}\sE_n$ are all isomorphisms in $\SH(k)$.
\end{proof}

Thus, for $\sE\in  \underline{W}_{*}\text{-Mod}$ fibrant, $\sE=((\sE_0,\sE_1,\ldots), \epsilon_n)$, we have the weak equivalence
\[
\phi_n:\sE_n\to \sE_0
\]
defined as the composition
\[
\sE_n\xrightarrow{\epsilon_W^n}\Omega^n_{\G_m}\sE_n\xrightarrow{\epsilon_*}\colim_p\Omega^p_{\G_m}\sE_p\cong \sE_0
\]
The $\phi_n$ define a  weak equivalence of spectra
\[
\phi:\sE\to \text{const}\sE_0
\]
giving us the commutative  diagram
\[
\xymatrix{ 
\Sigma^\infty_{\G_m}\sE_0\ar[r]^{can}\ar[d]_{\epsilon_W}&\sE\ar[dl]^\phi\\\text{const}\sE_0}.
\]
This shows that both $can$ and $\epsilon_W$ are isomorphisms in the homotopy category, and thus $\Sigma^\infty_{\G_m}$ and $\Omega^\infty_{\G_m}$ are inverse Quillen equivalences, giving us the desired equivalence on the homotopy categories.
Passing to the $\Q$-localization gives us the desired sequence of equivalences
\[
D_{\A^1}(k,\Q)^-\sim  \Hho \underline{W}^{cof}_{\Q*}\text{-Mod}\sim  \Hho \underline{W}_{\Q*}\text{-Mod}\sim DM_W(k)_\Q
\]

In the course of the proof of Theorem~\ref{thm:Equiv} we have also proved
\begin{corollary}\label{cor:Equiv}  The categories $DM_W(k)$ and  $\Hho  \underline{W}_{*}\text{-Mod}$ are equivalent as triangulated tensor categories. $\Hho  \underline{W}_{\Q*}\text{-Mod}$ and $\SH(k)^-_\Q$ are  equivalent as triangulated tensor categories. 
\end{corollary}

\begin{remark} One can rather easily compute the ``Witt motivic cohomology'' of some $X\in\Sm/k$. We define $M_W(X)=L_{\A^1,\G_m}\underline{W}(X)$. As  $\underline{W}$ is strictly $\A^1$-invariant and $\G_m$-stable, $\underline{W}$ is already an object of $DM_W(k)$. For $C$ in $DM_W(k)$, corresponding to $C_*$ in $\Hho \underline{W}\text{-Mod}$ via the equivalence in Corollary~\ref{cor:Equiv}, we have the object $C(q)$, $q\in\Z$,  corresponding to $\Sigma^q_{\G_m}C_*[-q]$.

Then
\begin{align*}
H^p(X,W(q))&:=\Hom_{DM_W(k)}(M_W(X), \underline{W}(q)[p])\\&\cong  \Hom_{DM_W(k)}(M_W(X), \underline{W}[p-q])\\
&\cong \Hom_{D(Sh^{Nis}_{\underline{W}}(\Sm/k))}(\underline{W}(X), \underline{W}[p-q])\\
&\cong \Hom_{D(Sh^{Nis}_{Ab}(\Sm/k))}(\Z(X), \underline{W}[p-q])\\
&\cong H^{p-q}_{Nis}(X, \underline{W})\\
&\cong H^{p-q}_{Zar}(X, \underline{W}).
\end{align*}
\end{remark}


\begin{thebibliography}{XXXX}

\bibitem[An15]{An15}
A.~Ananyevskiy,
\emph{Stable operations and cooperations in derived Witt theory with rational coefficients} preprint 2015.

\bibitem[An12]{An12}
A.~Ananyevskiy,
\emph{On the relation of special linear algebraic cobordism to Witt groups,}
arXiv:1212.5780 

\bibitem[Bal99]{Bal99}
P.~Balmer,
\emph{Derived Witt groups of a scheme,}
J. Pure Appl. Algebra, 141 (1999), 101--129.

\bibitem[Br84]{Br84}
G.W.~Brumfiel,
\emph{Witt rings and K-theory,}
Rocky Mountain J. of Math. 14, no. 4 (1984), pp. 733--765.

 \bibitem[CD12]{CD}
 D.C.~Cisinski,  F.~D\'eglise,
 \emph{Triangulated categories of mixed motives}, preprint 2012 (version 3), 
\verb!arXiv:0912.2110v3 [math.AG]!.

\bibitem[HO14]{HO}
J.~Heller, K.~Ormsby,
\emph{Galois equivariance and stable motivic homotopy theory}.
Preprint 14 Jun 2014 \verb!arXiv:1401.4728 [math.AT]!

\bibitem[H01]{Hovey}
M.~Hovey,   \emph{Spectra and symmetric spectra in general model categories}. J. Pure Appl. Algebra 165 (2001), no. 1, 63--127. 

\bibitem[Jac15]{Jac15}
J.A.~Jacobson,
\emph{From the global signature to higher signatures}, preprint 2015 (version 2), 
arXiv:1411.0993v2.

\bibitem[Jar00]{Jardine00}
J.F.~Jardine,  {\sl Motivic symmetric spectra},
{Doc. Math.} {\bf 5} (2000) 445--553.

\bibitem[Lev13a]{Lev13a}
M.~Levine, 
\emph{Convergence of Voevodsky's slice tower}, Doc. Math. {\bf 18} (2013) 907--941.


\bibitem[Lev13b]{Lev13b}
M.~Levine,
\emph{A comparison of motivic and classical stable homotopy theories}, 
J. Topol. 7 (2014), no. 2, 327--362.


\bibitem[Mor04a]{Mor04a}
F.~Morel,
\emph{An introduction to $\A^1$-homotopy theory}, 
Contemporary developments in algebraic $K$-theory, 357--441, 
ICTP Lect. Notes, 2004

\bibitem[Mor04b]{Mor04b}
F.~Morel,
\emph{On the motivic $\pi_0$ of the sphere spectrum,}
Axiomatic, enriched and motivic homotopy theory (2004), 219--260

\bibitem[Mor12]{Mor12}
F.~Morel,
\emph{$\A^1$-algebraic topology over a field,}
 Lecture Notes in Mathematics, 2052. Springer, Heidelberg, 2012.

\bibitem[MV99]{MV99}
F.~Morel, V.~Voevodsky,
\emph{$\A^1$-homotopy theory of schemes,}
Publ. Math. IHES, 90 (1999), 45--143.


\bibitem[O{\O}14]{OO}
K.~M.~Ormsby, P.~A.~{\O}stv{\ae}r, 
\emph{Stable motivic $\pi_1$ of low-dimensional fields}. Adv. Math. 265 (2014), 97--131.

\bibitem[PW10]{PW10}
I.~Panin and C.~Walter,
\emph{On the motivic commutative spectrum BO,}
arXiv:1011.0650.

\bibitem[RO08]{RO} 
O.~R\"ondigs, O., P.~A.~{\O}stv{\ae}r,  {\sl Modules over motivic cohomology}. Adv. Math. {\bf 219} (2008), no. 2, 689--727.

\bibitem[SS00]{SS}
S.~Schwede, B.E.~Shipley, 
\newblock{\sl Algebras and modules in monoidal model categories}. Proc. London
Math. Soc. (3) {\bf 80} (2000), no. 2, 491--511.

\bibitem[S51]{Serre}
J.~-P.~Serre,
\emph{Homologie singuli\`ere des espaces fibr\'es. III. Applications homotopiques}.  C. R. Acad. Sci. Paris 232, (1951). 142--144.
\end{thebibliography}
\end{document}